\documentclass[10pt, a4paper]{article}
\usepackage{amssymb} 
\usepackage{amsmath} 
\usepackage{amsthm} 

\newtheorem{theorem}{Theorem}

\newtheorem{lemma}[theorem]{Lemma}
\newtheorem{corollary}[theorem]{Corollary}

\newcommand{\eps}{\varepsilon}

\newcommand{\weight}{w}
\DeclareMathOperator{\rank}{rk}
\DeclareMathOperator{\PWE}{PWE}
\DeclareMathOperator{\PWV}{PWV}
\DeclareMathOperator{\subscriptsbf}{}
\newcommand{\Msbf}{{\mathcal{M}_{\subscriptsbf}}}
\newcommand{\Psbf}{{P_{\subscriptsbf}}}
\newcommand{\Osbf}{{\Omega_{\subscriptsbf}}}
\newcommand{\pisbf}{{\pi_{\subscriptsbf}}}
\DeclareMathOperator{\subscriptssf}{}
\newcommand{\Mssf}{{{\mathcal{M}}_{\subscriptssf}'}}
\newcommand{\Pssf}{{P_{\subscriptssf}'}}
\newcommand{\Ossf}{{\Omega_{\subscriptssf}'}}
\newcommand{\pissf}{{\pi_{\subscriptssf}'}}
\DeclareMathOperator{\cp}{cp}
\DeclareMathOperator{\lw}{lw}
\DeclareMathOperator{\pw}{pw}
\DeclareMathOperator{\tw}{tw}
\DeclareMathOperator{\vs}{vs}
\DeclareMathOperator{\inj}{inj}

\title{Rapid mixing of subset Glauber dynamics on graphs of bounded tree-width}

\author{Magnus Bordewich \and Ross J. Kang}

\begin{document}

\maketitle

\begin{abstract}
Motivated by the `subgraphs world' view of the ferromagnetic Ising model, we develop a general approach to studying mixing times of Glauber dynamics based on subset expansion expressions for a class of graph polynomials.  With a canonical paths argument, we demonstrate that the chains defined within this framework mix rapidly upon graphs of bounded tree-width. This extends known results on rapid mixing for the Tutte polynomial, the adjacency-rank ($R_2$-)polynomial and the interlace polynomial.

Keywords: Markov chain Monte Carlo, subset expansion, graph polynomials, tree-width, canonical paths, Tutte polynomial, interlace polynomial, random cluster model.
\end{abstract}

\section{Introduction}

We analyse a subset-sampling Markov chain on simple graphs that is derived from certain graph functions --- usually, in fact, graph polynomials. We show that this chain mixes rapidly on graphs of constant tree-width.

Throughout the paper, the graph functions $\mathcal{P}$ we consider are formulated using subset expansion\footnote{The term `subset expansion' was coined by Gordon and Traldi~\cite{GoTr90}, though it is a special type of `states model expansion' which is commonly used in physics.}.
An {\em edge subset expansion formula for $\mathcal{P}$} is written as follows: for any simple graph $G = (V,E)$,
\begin{align}
\label{eqn:polyedge}
\mathcal{P}(G)
= \sum_{S\subseteq E}
\weight((V,S))
\end{align}
for some graph function $\weight$, where $(V,S)$ denotes the graph with vertex set $V$ and edge set $S$. 
If the function $\weight$ is non-negative, that is, $\weight(G) \ge 0$ for all graphs $G$, we refer to~\eqref{eqn:polyedge} as an {\em edge subset weighting for $\mathcal{P}$} and to $\weight$ as its {\em weight function}.  In fact, we shall need the weight function to be {\em positive} on all subgraphs --- from a statistical physics viewpoint, this results in a so-called `soft-core model'.

Before moving on, let us anchor the general formula~\eqref{eqn:polyedge} with an example that is prominent in statistical physics, theoretical computer science, and discrete probability.
The {\em partition function of the random cluster model} can be defined for any $G = (V,E)$ and parameters $q,\mu$ as
\begin{align}
\label{eqn:RCpoly}
Z_{RC}(G;q,\mu) := \sum_{S\subseteq E}q^{\kappa(S)}\mu^{|S|},
\end{align}
where $\kappa(S)$ is the number of components in $(V,S)$.  For more on the random cluster model, see an extensive treatise by Grimmett~\cite{Gri06}.
Notice that, if $q, \mu \ge 0$, then $\weight((V,S)):=q^{\kappa(S)}\mu^{|S|}$ provides an edge subset weighting for $Z_{RC}(G;q,\mu)$.
Under a suitable transformation, $Z_{RC}(G;q,\mu)$ is equivalent to the Tutte polynomial~\cite{Tut54}, defined for any $G = (V,E)$ and parameters $x,y$ as
\begin{align}
\label{eqn:Tuttepoly}
T(G;x,y) := \sum_{S\subseteq E}(x-1)^{r(E)-r(S)}(y-1)^{|S|-r(S)},
\end{align}
where $r(S)$ is the $\mathbb{F}_2$-rank of the incidence matrix for $(V,S)$.
A wealth of combinatorial and structural information can be obtained from evaluations of this function. Indeed, this polynomial has a remarkable universality property, which informally speaking says that it subsumes any graph invariant that can be computed by deletion and contraction of edges~\cite{OxWe79}, cf.~\cite{WeMe00}.  In addition, the Tutte polynomial specialises to several key univariate graph polynomials, including the chromatic polynomial of  Birkhoff~\cite{Bir12}. 
It specialises to the Jones polynomial in knot theory~\cite{Jon85}.
By its connection with the random cluster model, it also generalises the partition functions of the Ising~\cite{Isi25} and Potts~\cite{Pot52} models\footnote{If $x,y\ge1$ or $q,\mu\ge0$, then, respectively, $T(G;x,y)$ or $Z_{RC}(G;q,\mu)$ generalise the partition functions of the {\em ferromagnetic} Ising and Potts models.}. Consult the monograph of Welsh~\cite{Wel93} for more on these crucial connections.
In addition to $Z_{RC}(G;q,\mu)$ and $T(G;x,y)$, we shall highlight a few other specific polynomials from the literature, but for a broad account of the development of graph polynomials, consult the recent surveys by Makowsky~\cite{Mak08} and by Ellis-Monaghan and Merino~\cite{ElMe11a,ElMe11b}.

It was shown in 1990 by Jaeger, Vertigan and Welsh~\cite{JVW90} that, in general, for fixed (rational) values of $x$ and $y$, the evaluation of $T(G;x,y)$ is \#P-hard, except on a few special points and curves in the $(x,y)$-plane.  As a result, there have been substantial efforts since then to pin down the {\em approximation complexity} of computing $T(G;x,y)$. For large swaths of the $(x,y)$-plane, it is now known that the computation of $T(G;x,y)$ either does not admit a fully polynomial-time randomised approximation scheme (FPRAS) unless RP = NP, or is at least as hard as \#BIS (the problem of counting independent sets in bipartite graphs) under approximation-preserving reductions, cf.~Goldberg and Jerrum~\cite{GoJe10a}.
The sole positive approximation result applicable to general graphs is the breakthrough FPRAS by Jerrum and Sinclair~\cite{JeSi90,JeSi93} for the partition function of the ferromagnetic Ising model --- this corresponds to computation of $T(G;x,y)$ along the portion of the parabola $(x-1)(y-1)=2$ with $y > 1$.
Various approaches have given efficient approximations in some regions of the Tutte plane for specific classes of graphs --- cf.~e.g.~Alon, Frieze and Welsh~\cite{AFW94,AFW95}, Karger~\cite{Kar95,Kar01}, and Bordewich~\cite{Bor04}.
To obtain their seminal result, Jerrum and Sinclair used a Markov chain Monte Carlo (MCMC) method, a principal tool in the design of efficient approximation schemes for counting problems.
MCMC methods are widespread in computational physics, computational biology, machine learning, and statistics.
There have been steady advances in our understanding of such random processes and in showing how quickly they produce good approximations of useful probability distributions in huge, complex data sets.
See the lecture notes of Jerrum~\cite{Jer03} or a survey by Randall~\cite{Ran06} for an overview of the application of these techniques in theoretical computer science.

We postpone the precise statement of our main result, Theorem~\ref{thm:mainedge,tw}, as it requires a host of definitions, but here we give a cursory description.
In this paper, we are interested in the rate of convergence to stationarity of a natural Markov chain closely associated to a subset weighting of $\mathcal{P}$ (of form~\eqref{eqn:polyedge}), when some mild restriction is placed upon the weight function $\weight$. 
That restriction --- which we have dubbed {\em $\lambda$-multiplicative} --- is described in Subsection~\ref{subsec:mild}: for now, we remark that some important graph polynomials and partition functions from statistical physics (e.g.~$Z_{RC}(G;q,\mu)$ and $T(G;x,y)$) obey it.
The state space of our chain is the set of all edge subsets, upon which we have set up a MCMC method using Glauber dynamics~\cite{Gla63}.
Each possible transition in the chain is either the addition or deletion of exactly one edge to/from the subset and the transition probabilities are defined according to the weights $\weight((V,S))$, subject to a Metropolis-Hastings filter~\cite{Has70,MRRTT53}\footnote{A Metropolis-Hastings filter is applied in order to ensure that the resulting process is a reversible Markov chain and thus guaranteed to converge to a unique stationary distribution with state probabilities proportional to the weight.}.
Our main finding is that on graphs of bounded tree-width this Markov chain converges to the stationary distribution in time that is polynomial in the number of vertices of the graph.

Our approach is inspired in part by the `subgraphs world' in which Jerrum and Sinclair~\cite{JeSi90,JeSi93} designed their FPRAS for the partition function of the ferromagnetic Ising model.
It is also motivated by recent work of Ge and \v{S}tefankovi\v{c}~\cite{GeSt09+,GeSt10}, who introduced the $R_2$-polynomial in an attempt to devise a FPRAS for \#BIS.  
Their {\em adjacency-rank polynomial} is defined for any $G = (V,E)$ and parameters $q,\mu$ as
\begin{align}
\label{eqn:R2poly}
R_2(G;q,\mu) := \sum_{S\subseteq E}q^{\rank_2(S)}\mu^{|S|},
\end{align}
where $\rank_2(S)$ is the $\mathbb{F}_2$-rank of the {\em adjacency} matrix for $(V,S)$.
Using a combinatorial interpretation of $\rank_2$ applicable only to bipartite graphs, they showed that the edge subset Glauber dynamics (using the weighting in~\eqref{eqn:R2poly}) mixes rapidly on trees. 
They conjectured that the chain mixes rapidly on all bipartite graphs, cf.~Conjecture~1 in~\cite{GeSt09+}. In addition, Ge and \v{S}tefankovi\v{c} showed that the Markov chain for the (soft-core) random cluster model --- i.e.~weighted according to~\eqref{eqn:RCpoly} --- mixes rapidly upon graphs of bounded tree-width. We have extended both of these results under a unified framework. In particular, we show that the $R_2$-polynomial fits in our framework without recourse to the combinatorial interpretation for bipartite graphs, and hence that the Markov chain for the $R_2$-polynomial mixes rapidly upon all graphs of bounded tree-width.
We also remark here that the conjectured rapid mixing of this chain on all bipartite graphs was disproved by Goldberg and Jerrum~\cite{GoJe10b}.

The polynomials and Markov chains that we capture in our framework are defined for any graph; however, we obtain rapid mixing results only on graphs of constant tree-width.
For brevity, we will not define tree-width here, but merely say that it is an essential concept in structural graph theory and parameterised complexity --- see modern surveys on the topic by Bodlaender~\cite{Bod06} and Hlin\v{e}n\'y et al.~\cite{HOSG08}. 
The restriction of tree-width is commonly used in graph algorithms to reduce the complexity of a computationally difficult problem, usually by way of dynamic programming.
For example, it is already known that many of the polynomials covered here can be evaluated efficiently for graphs of bounded tree-width.  Independently, Andrzejak~\cite{And98} and Noble~\cite{Nob98} exhibited polynomial-time algorithms to compute the Tutte polynomial of graphs with bounded tree-width.  Works of Makowsky and Mari\~{n}o~\cite{MaMa03} and Noble~\cite{Nob09} have significantly generalised this, in the former case, to a wide array of polynomials under the framework of monadic second order logic (MSOL), and, in the latter case, to the so-called $U$-polynomial~\cite{NoWe99}, a polynomial that includes not only the Tutte polynomial but also a powerful type of knot invariant as a special case.

Even though many of the polynomials we refer to can be computed exactly in polynomial time for graphs of bounded tree-width, it remains of interest to show that the associated Glauber dynamics is rapidly mixing. One hope is that for some polynomials the chain mixes rapidly for a wider class of graphs. There have been significant and concerted endeavours by researchers spanning physics, computer science and probability to determine the mixing properties of Glauber dynamics on many related Markov chains.  Spin systems have been of particular interest; indeed, the main thrust of the work of Jerrum and Sinclair~\cite{JeSi90,JeSi93} was to tackle the partition function for the `spins world' of the ferromagnetic Ising model (using a translation to the rapidly mixing `subgraphs world').  For more on the connections among the `spins world', the `subgraphs world' and the `random cluster world', see the recent work of Huber~\cite{Hub09+}.
We note that many recent projects on spin systems have been restricted to trees or tree-like graphs, cf.~e.g.~\cite{BKMP05,DeMo10,DLP10,GJK10,MSW07,TVVY10}.

Our primary focus in this paper is to establish results for polynomials defined according to {\em edge} subset expansion, but we can also extend our methodology to polynomials defined according to {\em vertex} subset expansion, which may be viewed as the `induced subgraphs world'. 
To our knowledge, this form of Markov chain has not been greatly examined, but it handles one important graph polynomial that was recently introduced by Arratia, Bollob\'as and Sorkin~\cite{ABS04b}:
the {\em bivariate interlace polynomial}, defined for any graph $G = (V,E)$ and parameters $x,y$ as
\begin{align}
\label{eqn:interpoly}
q(G;x,y) := \sum_{S\subseteq V}(x-1)^{\rank_2(S)} (y-1)^{|V|-\rank_2(S)},
\end{align}
where $\rank_2(S)$ is the $\mathbb{F}_2$-rank of the adjacency matrix for $G[S]$.
This polynomial specialises to the independence polynomial and is intimately related to Martin polynomials~\cite{AiHo04}.
Just as for the Tutte polynomial, computation of the bivariate interlace polynomial is \#P-hard in almost the entire plane, as was shown by Bl\"aser and Hoffmann~\cite{BlHo08}.
The multivariate interlace polynomial, a generalisation of the interlace polynomial, can be evaluated efficiently for graphs of bounded tree-width, cf.~Courcelle~\cite{Cou08} and Bl\"aser and Hoffmann~\cite{BlHo09,BlHo11+}.  Subject to a condition on the vertex subset weightings, which we have called {\em vertex $\lambda$-multiplicativity}, we can establish rapid mixing for vertex subset Glauber dynamics on graphs of constant tree-width.

For all of our results, we need that the weight function is strictly positive for all (induced) subgraphs. Many of the classical enumeration polynomials such as the matching, independence, clique and chromatic polynomials are captured by the general polynomials that we mention as examples throughout this work. However, these are `hard-core models', in which some (induced) subgraphs have a zero weighting, and hence are not included in our approach. Many of these are evaluations that fall at the boundary of the regions that we can handle. For example, the Tutte polynomial evaluated at the point $(2,1)$ counts the number of forests of the graph. We have shown rapid mixing at all fixed points $(2,1+\delta$), for $\delta>0$, with a mixing time that depends on $\delta$. It would be interesting to consider whether the chains associated with these boundary points mix rapidly for graphs of bounded tree-width.

\medskip

The structure of this paper is as follows.  In the next section, we give the definitions that are necessary for a detailed description of the main theorem.  We give the main theorem in Section~\ref{sec:results} and then indicate some of its consequences.  We present the proofs in Section~\ref{sec:proof}.  In Section~\ref{sec:induced}, we extend our results to Glauber dynamics on vertex subsets, that is, on induced subgraphs.

\section{Definitions}\label{sec:definitions}

\subsection{$\lambda$-multiplicative weight functions}\label{subsec:mild}

In this subsection, we describe the condition we require on our graph functions $\mathcal{P}$.
This condition prescribes that the weight function is multiplicative with respect to the operation of disjoint graph union as well as ``nearly multiplicative'' with respect to the operation of composition via small vertex cuts.

We use the notation $\hat\lambda := \max\{\lambda,1/\lambda\}$. For a graph $G = (V,E)$, a vertex cut $K$ is said to separate sets $V_1$ and $V_2$ if $(V_1,K,V_2)$ is a partition of $V$ and there is no edge of $E$ that is incident to both a vertex of $V_1$ and a vertex of $V_2$.
A partition $(E_1,E_2)$ of $E$ is appropriate (for $K$) if $E_1$ has no edge adjacent to a vertex in $V_2$ and $E_2$ has no edge adjacent to a vertex in $V_1$.

For fixed $\lambda > 0$, we say that the weight function $\weight$ is {\em $\lambda$-multiplicative}, if for any $G = (V,E)$, any vertex cut $K$ that separates sets $V_1$ and $V_2$, and any appropriate partition $(E_1,E_2)$, we have
\begin{align}\label{eqn:edgekmult}
\hat\lambda^{-|K|} \le \frac{\weight((V_1\cup K,E_1))\weight((V_2\cup K,E_2))}{\weight(G)} \le \hat\lambda^{|K|}.
\end{align}
As mentioned above, if $\weight$ is $\lambda$-multiplicative, then it follows that $\weight$ is multiplicative with respect to disjoint union (by taking $K = \emptyset$); furthermore, taking $V_2 = \emptyset$ implies that the addition or deletion of a few edges in the graph does not change $\weight$ wildly.

\subsection{Examples of valid polynomials}\label{subsec:kmultex}

In this subsection, we emphasise specific examples of edge subset weightings and justify that their weight functions are $\lambda$-multiplicative.

Let $G = (V,E)$ be any graph, $K$ be any vertex cut that separates vertex subsets $V_1$ and $V_2$, and $(E_1,E_2)$ be any appropriate partition. 
We define $G'$ to be the disjoint union of graphs $(V_1\cup K,E_1)$ and $(V_2\cup K,E_2)$. We could imagine forming $G'$ from $G$ by splitting each vertex in $K$, taking incident edges in $E_1$ with one copy of the vertex and those in $E_2$ with the other.
It is trivial to verify multiplicativity with respect to disjoint union for each of the weight functions considered below.
Therefore, to establish $\lambda$-multiplicativity for these weight functions, it will suffice to verify that 
$\hat\lambda^{-|K|} \le \weight(G')/\weight(G) \le \hat\lambda^{|K|}$.

First, we observe that the partition function of the \emph{random cluster model} for $q,\mu > 0$ satisfies the condition. Recalling~\eqref{eqn:RCpoly},
the relevant weight function is $\weight((V,S)) := q^{\kappa(S)}\mu^{|S|}$.
To handle the $\mu^{|S|}$ factor, note that the graphs $G$ and $G'$ have the same number of edges.
For the $q^{\kappa(S)}$ factor, the number of components in $G'$ can be at most $\kappa(G)+|K|$ since $G'$ can be obtained by splitting $|K|$ vertices of $G$. Thus, $\weight$ is $\lambda$-multiplicative if we take $\lambda := q$. 

This can also be seen in the context of the \emph{Tutte polynomial} when $x,y>1$. Recalling~\eqref{eqn:Tuttepoly},
the relevant weight function is $\weight((V,S)) := (x-1)^{r(E)-r(S)}(y-1)^{|S|-r(S)}$.
As before, it is easy to take care of the $(x-1)^{r(E)}(y-1)^{|S|}$ factor.
For the remaining $((x-1)(y-1))^{-r(S)}$ factor, it is enough to observe that the incidence matrix of $G$ may be obtained from the incidence matrix of $G'$ as follows. The matrix for $G'$ has two rows for each of the vertices in $K$, one from $(V_1\cup K,E_1)$ and one from $(V_2\cup K,E_2)$. If we replace one of these two rows with the sum of the two rows, we do not alter the rank; if we then delete the other of the two rows, we change the rank by at most $1$. Doing this for each vertex in $K$, we obtain the incidence matrix for $G$, at a total change in the rank $r$ of the incidence matrix of at most $|K|$. Thus, $\weight$ is $\lambda$-multiplicative if we take $\lambda := (x-1)(y-1)$.

Next, we see that the {\em adjacency-rank polynomial} of Ge and \v{S}tefankovi\v{c}~\cite{GeSt09+} satisfies the condition if $q,\mu>0$. Recalling~\eqref{eqn:R2poly}, the relevant weight function is  $\weight((V,S)) := q^{\rank_2(S)}\mu^{|S|}$.
As before, it is simple to handle the $\mu^{|S|}$ factor.
For the $q^{\rank_2(S)}$ factor, we note that the adjacency matrix of $G$ may be formed from the adjacency matrix of $G'$ by $|K|$ row additions, followed by $|K|$ column additions and finally the deletion of $|K|$ rows and $|K|$ columns. Since we must delete both rows and columns, the rank $\rank_2$ of the adjacency matrix may change by up to $2|K|$. Thus, in this case, $\weight$ is $\lambda$-multiplicative when taking  $\lambda := q^2$.

Now, consider the {\em multivariate Tutte polynomial} as formulated by Sokal~\cite{Sok05}, defined for any graph $G = (V,E)$ and parameters $q,\vec{v}=\{v_e\}_{e\in E}$ by
\begin{align}
\label{eqn:multiTuttepoly}
Z_{Tutte}(G;q,\vec{v}) := \sum_{S\subseteq E}q^{\kappa(S)}\prod_{e\in S}v_e.
\end{align}
Under this expansion, $\weight := q^{\kappa(S)}\prod_{e\in S}v_e$ is an edge subset weight function if $q > 0$ and $v_e>0$ for any $e \in E$ are fixed. We can handle the $q^{\kappa(S)}$ factor as we did for the random cluster model partition function. For the $\prod_{e\in S}v_e$ factor, observe that $G$ and $G'$ have the same set of edges. Thus, $\weight$ is $\lambda$-multiplicative when taking $\lambda := q$.

Last, we discuss the {\em $U$-polynomial} of Noble and Welsh~\cite{NoWe99}, defined for any graph $G = (V,E)$ and parameters $y,\vec{x}=\{x_i\}_{i=1}^{|V|}$ by
\begin{align}
\label{eqn:Upoly}
U(G;\vec{x},y) := \sum_{S\subseteq E}(y-1)^{|S|-r(S)}\prod_{i=1}^{|V|}{x_i}^{\kappa(i,S)},
\end{align}
where $\kappa(i,S)$ denotes the number of components of order $i$ in $(V,S)$.
If $y>1$ and $x_i>0$ for all $i$, then $\weight((V,S)) := (y-1)^{|S|-r(S)}\prod_{i=1}^{|V|}{x_i}^{\kappa(i,S)}$ gives an edge subset weighting.
The $(y-1)^{|S|-r(S)}$ factor can be handled as above.
For the $\prod_{i=1}^{|V|}{x_i}^{\kappa(i,S)}$ factor, observe that $\sum_i |\kappa(i,G)-\kappa(i,G')|$ is at most $3|K|$, since, if we obtain $G'$ by splitting the vertices in $K$, each time we split a vertex we either change the size of a single component or split a single component into two smaller components. Thus, taking $x' := \max_i \max\{x_i,x_i^{-1}\}$ and $y' := \max\{y-1,(y-1)^{-1}\}$, we see that $\weight$ is $\lambda$-multiplicative when taking $\lambda := y'{x'}^3$.

\subsection{Glauber dynamics for edge subsets}\label{subsec:chains}

\noindent
In this subsection, we define the Markov chain associated with the edge subset expansion formula for $\mathcal{P}$. From the formulation in~\eqref{eqn:polyedge}, the {\em single bond flip chain} $\Msbf$ on a given graph $G = (V,E)$ is defined as follows.  We start with an arbitrary subset $X_0 \subseteq E$ and repeatedly generate $X_{t+1}$ from $X_t$ by running the following experiment.
\begin{enumerate}
\item Pick an edge $e\in E$ uniformly at random and let $S = X_t \oplus \{e\}$.
\item Set $X_{t+1} = S$ with probability $\frac12\min\left\{1,\weight((V,S))/\weight((V,X_t))\right\}$ and with the remaining probability set $X_{t+1} = X_t$.
\end{enumerate}
By convention, we denote the state space of $\Msbf$ by $\Osbf$ (i.e.~$\Osbf = 2^E$) and its transition probability matrix by $\Psbf$.
With standard arguments, it can be shown that $\Msbf$ is a reversible Markov chain 
that has a unique stationary distribution $\pisbf$ satisfying $\pisbf(S) \propto \weight((V,S))$. 
Hence, we may use $\Msbf$ as a Markov chain in MCMC sampling for the following problem.

\medskip
\noindent
$\PWE(\mathcal{P})$: \textsc{$\mathcal{P}$-weighted Edge Subsets}\\
\indent Input: a graph $G = (V, E)$.\\
\indent Output: a subset $S\subseteq E$ with probability $\weight((V,S))/\mathcal{P}(G)$.
\medskip

The term \emph{rapidly mixing} applies to a Markov chain that quickly converges to its stationary distribution. We make this precise here. 
The {\em total variation distance} $\Vert\nu - \nu'\Vert_{TV}$ between two probability distributions $\nu$ and $\nu'$ is defined by
$\Vert\nu - \nu'\Vert_{TV} = \frac12\sum_{H\in\Omega}|\nu(H)-\nu'(H)|$.
For $\eps>0$, the {\em mixing time} of a Markov chain $\mathcal{M}$ (with state space $\Omega$, transition matrix $P$ and stationary distribution $\pi$) is defined as
\begin{align*}
\tau(\eps) := \max_{H\in\Omega}\{\min\{t \ | \ \Vert P^t(H,\cdot)-\pi(\cdot) \Vert_{TV} \le \eps\}\}.
\end{align*}
In this paper, we shall say that a chain $\mathcal{M}$ {\em mixes rapidly} if, for any fixed $\eps$, $\tau(\eps)$ is (upper) bounded by a polynomial in the number of vertices of the input graph. This definition for rapid mixing is the one more commonly used in theoretical computer science, whereas often in statistical physics or discrete probability a stricter $O(n \log n)$ bound is mandated.

\section{Results}\label{sec:results}

We are now prepared to state the main theorem.

\begin{theorem}\label{thm:mainedge,tw}
Let $G=(V,E)$ where $|V| = n$.  If $\weight$ is $\lambda$-multiplicative for some $\lambda > 0$, then the mixing time of $\Msbf$ on $G$ satisfies
\begin{align*}
\tau(\eps) = O\left(n^{4+4(\tw(G)+1)|\log\lambda|}
\log(1/\eps)
\right)
\end{align*}
(where $\tw(G)$ denotes the tree-width of $G$).
\end{theorem}

In Subsection~\ref{subsec:kmultex}, we noted some examples of polynomials that have $\lambda$-multiplicative weight functions; thus, Theorem~\ref{thm:mainedge,tw} implies that each of their associated Glauber dynamics on edge subsets is rapidly mixing upon graphs of bounded tree-width.

\begin{corollary}
\label{cor:consequences}
Let $G = (V,E)$ where $|V| = n$.
In the following list, we state conditions on the parameters which guarantee rapid mixing of the single bond flip chain on $G$ associated with the stated polynomial and weighting.  We also state the mixing time bound.
\begin{enumerate}
\item For fixed $q,\mu>0$ and the weighting~\eqref{eqn:RCpoly} of $Z_{RC}(G;q,\mu)$, the mixing time satisfies
\begin{align*}
\tau(\eps) = O\left(n^{4+4(\tw(G)+1)|\log q|}
\log(1/\eps)
\right).
\end{align*}
Equivalently, for fixed $x,y>1$ and the weighting~\eqref{eqn:Tuttepoly} of $T(G;x,y)$, the mixing time satisfies
\begin{align*}
\tau(\eps) = O\left(n^{4+4(\tw(G)+1)|\log ((x-1)(y-1))|}
\log(1/\eps)
\right).
\end{align*}
\item For fixed $q,\mu>0$ and the weighting~\eqref{eqn:R2poly} of $R_2(G;q,\mu)$, the mixing time satisfies
\begin{align*}
\tau(\eps) = O\left(n^{4+8(\tw(G)+1)|\log q|}
\log(1/\eps)
\right).
\end{align*}
\item For fixed $q>0$ and $v_e>0$ for all $e$ and the weighting~\eqref{eqn:multiTuttepoly} of $Z(G;q,\vec{v})$, the mixing time satisfies
\begin{align*}
\tau(\eps) = O\left(n^{4+4(\tw(G)+1)|\log q|}
\log(1/\eps)
\right).
\end{align*}
\item For fixed $y>1$ and $x_i>0$ for all $i$ and the weighting~\eqref{eqn:Upoly} of $U(G;\vec{x},\mu)$, the mixing time satisfies
\begin{align*}
\tau(\eps) = O\left(n^{4+4(\tw(G)+1)\left|\log \left(y'{x'}^3\right)\right|}
\log(1/\eps)
\right)
\end{align*}
where $x' = \max_i \max\{x_i,x_i^{-1}\}$ and $y' = \max\{y-1,(y-1)^{-1}\}$.
\end{enumerate}
\end{corollary}

Here, we remark that Ge and \v{S}tefankovi\v{c} obtained part~1 above and showed part~2 above in the special case of trees.  Parts~2--4 directly extend these findings, and our main theorem considerably broadens the scope of mixing time bounds for subset Glauber dynamics on graphs of bounded tree-width.

\section{Proofs}\label{sec:proof}

Let us first give an outline of the proof.

Although our main result is stated in terms of tree-width, we do not treat tree-width directly but instead use linear-width, a more restrictive width parameter introduced by Thomas~\cite{Tho96}. This strategy was also employed by Ge and \v{S}tefankovi\v{c} in the two specific cases mentioned above. For any graph $G=(V,E)$, an ordering  $(e_1,\dots,e_m)$ of $E$ has linear-width at most $\ell$, if, for each $i \in \{1,\dots,m\}$, there are at most $\ell$ vertices that are incident to both an edge in $\{e_1,\dots,e_{i-1}\}$ and an edge in $\{e_i,\dots,e_m\}$. 
The {\em linear-width $\lw(G)$} of $G = (V,E)$ is the smallest integer $\ell$ such that there is an ordering of $E$ with linear-width at most $\ell$. 
The motive for using linear-width is that it implies an ordering of the edges which we can then use to define canonical paths between pairs of edge subsets. Then we show that $\lambda$-multiplicativity is the general condition under which we can  bound the congestion of these canonical paths.  The use of canonical paths is a standard technique for obtaining a bound on the mixing time of MCMC methods --- see the lecture notes of Jerrum~\cite{Jer03} for an expository account of this approach.

The key property we require that relates the linear-width of $G$ to the more well-studied parameters path-width $\pw(G)$ and tree-width $\tw(G)$ of $G$ is the following set of inequalities, details of which can be found in Bodlaender~\cite{Bod88}, Chung and Seymour~\cite{ChSe89}, Fomin and Thilikos~\cite{FoTh06}, Ge and \v{S}tefankovi\v{c}~\cite{GeSt09+}, and Korach and Solel~\cite{KoSo93}.
For any graph $G$ on $n$ vertices,
\begin{align}
\label{eqn:width}
\pw(G) \le \lw(G) \le \pw(G) + 1 \le (\tw(G)+1)(\lfloor \log_2 n \rfloor+1)+1.
\end{align}

We follow a canonical paths strategy to bound the mixing time of $\Msbf$.
Given $G = (V,E)$, let $\sigma = (e_1,\dots,e_m)$ be an ordering of $E$.  Given $I,F \in \Osbf$, let $I\oplus F$ denote the symmetric difference of $I$ and $F$, let $\sigma[I\oplus F] := (e_{i_1},\dots,e_{i_k})$ denote the restriction of $\sigma$ to $I\oplus F$ (that is, $\{e_{i_1},\dots,e_{i_k}\} = I\oplus F$ and $i_1<\cdots<i_k$), and let $\gamma_{\sigma,I\to F}$ denote the {\em canonical path} from $I$ to $F$, defined as
\begin{align*}
\gamma_{\sigma,I\to F} := (H_0,\dots,H_k),
\end{align*}
where $H_0 = I$, $H_j=H_{j-1}\oplus \{e_{i_j}\}$ for all $j\in\{1,\dots,k\}$ (and hence $H_k =F$).  Let $\Gamma_\sigma = \{\gamma_{\sigma,I\to F} \ | \ I, F \in \Omega\}$.

To bound the mixing time of $\Msbf$, we will, for some appropriately chosen $\sigma$, bound the {\em congestion} $\varrho(\Gamma_\sigma)$ of the canonical paths, which is defined by
\begin{align}
\label{eqn:congestion}
\varrho(\Gamma_\sigma) :=
\max_{\substack{(H,H'):\\P(H,H')>0}}
\left\{
\frac{1}{\pi(H)P(H,H')}
\sum_{\substack{I,F:\\(H,H')\in\gamma_{\sigma,I\to F}}}
\pi(I)\pi(F)|\gamma_{\sigma,I\to F}|
\right\},
\end{align}
where $|\gamma_{\sigma,I\to F}|$ denotes the length of the path $\gamma_{\sigma,I\to F}$.  
The mixing time can then be bounded using the following inequality of Sinclair~\cite{Sin92} --- see also Diaconis and Stroock~\cite{DiSt91}:
for any set $\Gamma$ of canonical paths,
\begin{align}
\label{eqn:mixingcongestion}
\tau(\eps) \le \max_{H\in\Omega}\left\{\varrho(\Gamma)\cdot\left(\log\frac1{\pi(H)}+\log \frac1\eps\right) \right\}.
\end{align}

The remainder of the section is devoted to showing the following.

\begin{theorem}\label{thm:mainedge,lw}
Suppose $G = (V,E)$ has linear-width $\ell$ and let $\sigma = (e_1,\dots,e_m)$ be an ordering of $E$ with linear-width at most $\ell$.  If $\weight$ is $\lambda$-multiplicative for some $\lambda > 0$, then $\varrho(\Gamma_\sigma) \le 2m^2\hat\lambda^{4\ell}$.
\end{theorem}

\noindent
Theorem~\ref{thm:mainedge,lw} immediately implies a good mixing time bound for the Markov chain $\Msbf$ and hence Theorem~\ref{thm:mainedge,tw} follows.

\begin{corollary}\label{cor:lw}
Let $G=(V,E)$ where $|E| = m$.  If $\weight$ is $\lambda$-multiplicative for some $\lambda > 0$, then the mixing time of $\Msbf$ on $G$ satisfies
\begin{align*}
\tau(\eps) = O\left(m^2 \hat\lambda^{4\lw(G)}
\log(1/\eps)
\right).
\end{align*}
\end{corollary}

\begin{proof}
Substitute the congestion bound of Theorem~\ref{thm:mainedge,lw} into the inequality~\eqref{eqn:mixingcongestion}.
\end{proof}

\begin{proof}[Proof of Theorem~\ref{thm:mainedge,tw}]
Substitute the upper bound on $\lw(G)$ in~\eqref{eqn:width} into Corollary~\ref{cor:lw}.
\end{proof}

In the proof of Theorem~\ref{thm:mainedge,lw}, we will need the following lemma.

\begin{lemma}\label{lem:fdiffedge}
Suppose $G = (V,E)$ has linear-width $\ell$ and let $\sigma = (e_1,\dots,e_m)$ be an ordering of $E$ with linear-width at most $\ell$.  Suppose $I,F\in\Osbf$ and $H$ is on $\gamma_{\sigma,I\to F}$.
If $\weight$ is $\lambda$-multiplicative for some $\lambda > 0$, then
\begin{align*}
\frac{\weight((V,I))\weight((V,F))}{\weight((V,H))\weight((V,C))} \le \hat\lambda^{4\ell},
\end{align*}
where $C = I \oplus F \oplus H$.
\end{lemma}

\begin{proof}
Since $H$ is on $\gamma_{\sigma,I\to F}$, we may assume that $H = H_j$ for some $j\in\{0,\dots,k\}$.  Let $Q = \{e_1,\dots,e_{i_j}\}$ and $\overline Q = E \setminus Q$. Then
\begin{align}
H = (F\cap Q)\cup(I\cap \overline Q) \quad \text{and} \quad C = (I\cap Q)\cup(F\cap \overline Q). \label{eqn:HCedge}
\end{align}
We can partition $V$ into three sets as follows. Let $V_1$ denote the set of vertices that are incident only to edges in $Q$; let $V_2$ denote the set of vertices that are incident only to edges in $\overline Q$; let $K$ denote the set of remaining vertices, that is, the set of vertices incident to edges in $Q$ and $\overline Q$. Note that $|K|$ is at most the linear-width $\ell$.

No vertex $v_1$ of $V_1$ is adjacent to a vertex $v_2$ of $V_2$, as otherwise the edge between them would simultaneously be in $Q$ and $\overline Q$. 
This implies that $K$ is a vertex cut separating $V_1$ and $V_2$ with respect to $G$, and also with respect to the graphs $(V,I)$, $(V,F)$, $(V,H)$, $(V,C)$.
Furthermore, $(I\cap Q,I\cap\overline Q)$, $(F\cap Q,F\cap\overline Q)$, $(H\cap Q,H\cap\overline Q)$, $(C\cap Q,C\cap\overline Q)$ are edge partitions that are appropriate for $K$.  
Therefore, by the fact that $\weight$ is $\lambda$-muliplicative and $|K| \le \ell$,
\begin{align*}
\hat\lambda^{-\ell} \le \frac{\weight((V_1\cup K,J\cap Q))\weight((V_2\cup K,J\cap\overline Q))}{\weight((V,J))} \le \hat\lambda^\ell \quad \text{for } J\in \{I,F,H,C\}. 
\end{align*}
By~\eqref{eqn:HCedge}, it follows that
\begin{align*}
(V_1\cup K,H\cap Q) &= (V_1\cup K,F\cap Q), \\
(V_2\cup K,H\cap\overline Q) &= (V_2\cup K,I\cap\overline Q), \\
(V_1\cup K,C\cap Q) &= (V_1\cup K,I\cap Q)\quad\text{ and} \\
(V_2\cup K,C\cap\overline Q) &= (V_2\cup K,F\cap\overline Q).
\end{align*}
Now, letting $r$ be
\begin{align*}
\weight((V_1\cup K,I\cap Q))\weight&((V_2\cup K,I\cap\overline Q))
\weight((V_1\cup K,F\cap Q))\weight((V_2\cup K,F\cap\overline Q))
\end{align*}
we obtain
\begin{align*}
\frac{\weight((V,I))\weight((V,F))}{r} \le \hat\lambda^{2\ell} 
\quad \text{and} \quad
\frac{r}{\weight((V,H))\weight((V,C))} \le \hat\lambda^{2\ell}, 
\end{align*}
whereby the lemma easily follows.
\end{proof}

\begin{proof}[Proof of Theorem~\ref{thm:mainedge,lw}]
Let $(H,H')\in\Osbf\times\Osbf$ such that $\Psbf(H,H')>0$.  We will bound the expression within the $\max$ of the definition for $\varrho(\Gamma_\sigma)$.
We let $\hat{H}=H$ if $\pisbf(H)\le\pisbf(H')$ and $\hat{H}=H'$ otherwise.  
Denote by $\cp(H,H')$ the set of pairs $(I,F)$ such that $(H,H')\in\gamma_{\sigma,I\to F}$.  We define the function $\inj: \cp(H,H')\to\Osbf$ by $(I,F) \to I\oplus F\oplus \hat{H}$.  Observe that $\inj$ is an injection, for, given $J\in \Osbf$, we can determine the unique $(I,F)$ such that $\inj(I,F) = J$ by first computing $J\oplus\hat{H}=I\oplus F$ and then using the ordering $\sigma$ to recover $I$ and $F$.  Since $\weight$ is $\lambda$-multiplicative, we have by Lemma~\ref{lem:fdiffedge} that
\begin{align}
\frac{\weight((V,I))\weight((V,F))}{\weight((V,\hat{H}))\weight((V,\inj(I,F)))} \le \hat\lambda^{4\ell}.
\label{eqn:fdiffedge}
\end{align}
Regardless of whether $\pisbf(H)\le\pisbf(H')$ or $\pisbf(H)>\pisbf(H')$, a brief calculation yields that
$\pisbf(H)\Psbf(H,H')
 = \pisbf(\hat{H})/(2m)$; thus,
\begin{align}
\frac{1}{\pisbf(H)\Psbf(H,H')}&\sum_{(I,F)\in\cp(H,H')}\pisbf(I)\pisbf(F)|\gamma_{\sigma,I\to F}|\nonumber\\
& = \frac{2m}{\pisbf(\hat{H})}\sum_{(I,F)\in\cp(H,H')}\pisbf(I)\pisbf(F)|\gamma_{\sigma,I\to F}|\nonumber\\
& \le \frac{2m^2}{\mathcal{P}(G)}\sum_{(I,F)\in\cp(H,H')}\frac{\weight((V,I))\weight((V,F))}{\weight((V,\hat{H}))} \label{line1}\\
& \le \frac{2m^2}{\mathcal{P}(G)}\sum_{(I,F)\in\cp(H,H')}\weight((V,\inj(I,F)))\hat\lambda^{4\ell} \label{line2}\\
& \le 2m^2\hat\lambda^{4\ell}, \label{line3}
\end{align}
where~\eqref{line1} follows from the facts $|\gamma_{\sigma,I\to F}|\le m$ and $\pisbf(S) = \weight((V,S))/\mathcal{P}(G)$,
\eqref{line2} follows from~\eqref{eqn:fdiffedge},
and~\eqref{line3} follows from the fact that $\inj$ is an injection.  Then, substituting the bound~\eqref{line3} into~\eqref{eqn:congestion}, we obtain $\varrho(\Gamma_\sigma) \le 2m^2\hat\lambda^{4\ell}$, as claimed.
\end{proof}

\section{Vertex subset Glauber dynamics for bounded tree-width}\label{sec:induced}

Until now, we had been considering edge subsets (subgraphs) and Glauber transitions which change one edge at a time.  In this section, we modify our methods to treat vertex subsets (induced subgraphs) and transitions that involve one vertex at a time --- each such transition can affect many edges, up to the maximum degree of $G$.  We sketch how to obtain rapid mixing for this process upon graphs of bounded tree-width still with only a modest condition on the base graph polynomials.

A {\em vertex subset expansion formula for $\mathcal{P}$} is written as follows: for any simple graph $G = (V,E)$, 
\begin{align}
\label{eqn:polyvertex}
\mathcal{P}(G)
= \sum_{S\subseteq V}
\weight(G[S])
\end{align}
for some graph function $\weight$, where $G[S]$ denotes the subgraph of $G$ induced by $S$. 
If the function $\weight$ is non-negative, we refer to~\eqref{eqn:polyvertex} as an {\em vertex subset weighting for $\mathcal{P}$} and to $\weight$ as its {\em weight function}.  Again, for our results to hold, aside from some other constraints, we need the weight function to be positive on all induced subgraphs.

From the formulation in~\eqref{eqn:polyvertex}, we define the {\em single site flip chain} $\Mssf$ on a given graph $G = (V,E)$ as follows.  We start with an arbitrary subset $X_0 \subseteq V$ and repeatedly generate $X_{t+1}$ from $X_t$ by running the following experiment.
\begin{enumerate}
\item Pick a vertex $v\in V$ uniformly at random and let $S = X_t \oplus \{v\}$.
\item Set $X_{t+1} = S$ with probability $\frac12\min\left\{1,\weight(G[S])/\weight(G[X_t])\right\}$ and with the remaining probability set $X_{t+1} = X_t$.
\end{enumerate}
We denote the state space of $\Mssf$ by $\Ossf$ (i.e.~$\Ossf = 2^V$) and its transition probability matrix by $\Pssf$.
It can be shown that $\Mssf$ is a reversible Markov chain that has a unique stationary distribution $\pissf$ satisfying
$\pissf(S) \propto \weight(G[S])$.
Hence, we may use $\Mssf$ as a Markov chain in MCMC sampling for the following problem.
\medskip

\noindent
$\PWV(\mathcal{P})$: \textsc{$\mathcal{P}$-weighted Vertex Subsets}\\
\indent Input: a graph $G = (V, E)$.\\
\indent Output: a subset $S\subseteq V$ with probability $\weight(G[S])/\mathcal{P}(G)$.
\medskip

We now describe the condition required of the weight function $\weight$ in~\eqref{eqn:polyvertex}.
For fixed $\lambda > 0$, we say that the weight function $\weight$ is {\em vertex $\lambda$-multiplicative}, if for any $G = (V,E)$ and $K$ a vertex cut that separates sets $ V_1$ and $V_2$ with respect to $G$, we have
\begin{align}\label{eqn:vertexkmult}
\hat\lambda^{-|K|} \le \frac{\weight(G[V_1])\weight(G[V_2\cup K])}{\weight(G)} \le \hat\lambda^{|K|}.
\end{align}
Note that, if $\weight$ is vertex $\lambda$-multiplicative, then it follows that $\weight$ is multiplicative with respect to disjoint union by taking $K = \emptyset$; furthermore, taking $V_2 = \emptyset$ gives that the addition of a few vertices does not change $\weight$ wildly.

The main result of this section is the following.

\begin{theorem}\label{thm:mainvertex,tw}
Let $G=(V,E)$ where $|V| = n$.  If $\weight$ is vertex $\lambda$-multiplicative for some $\lambda > 0$, then the mixing time of $\Mssf$ on $G$ satisfies
\begin{align*}
\tau(\eps) = O\left(n^{2+4(\tw(G)+1)|\log\lambda|}
\log(1/\eps)
\right).
\end{align*}
\end{theorem}

\subsection{A sketch of the proof}

As before, we do not treat tree-width directly, but instead work with a different width parameter.
For any graph $G=(V,E)$, an ordering  $(v_1,\dots,v_n)$ of $V$ has vertex-separation at most $\ell$, if, for each $i \in \{1,\dots,n\}$, there are at most $\ell$ vertices in $\{v_1,\dots,v_{i-1}\}$ that are adjacent to a vertex in $\{v_i,\dots,v_n\}$. 
The {\em vertex-separation $\vs(G)$} of $G = (V,E)$ is the smallest integer $\ell$ such that there is an ordering of $V$ with vertex-separation at most $\ell$. 
It was shown by Kinnersley~\cite{Kin92} that the vertex-separation of $G$ satisfies $\vs(G) = \pw(G)$, and so the inequalities in~\eqref{eqn:width} remain relevant.

To bound the mixing time of $\Mssf$, we again follow a canonical paths argument.
Given $G = (V,E)$, let $\sigma = (v_1,\ldots,v_n)$ be an ordering of $V$.  Given $I,F \in \Ossf$, let $I\oplus F$ denote the symmetric difference of $I$ and $F$, let $\sigma[I\oplus F] := (v_{i_1},\dots,v_{i_k})$ denote the restriction of $\sigma$ to $I\oplus F$ (that is, $\{v_{i_1},\dots,v_{i_k}\} = I\oplus F$ and $i_1<\cdots<i_k$), and let $\gamma_{\sigma,I\to F}$ denote the {\em canonical path} from $I$ to $F$, defined as
\begin{align*}
\gamma_{\sigma,I\to F} := (H_0,\dots,H_k),
\end{align*}
where $H_0 = I$, $H_j=H_{j-1}\oplus \{v_{i_j}\}$ for all $j\in\{1,\dots,k\}$ (and hence $H_k =F$).  Let $\Gamma_\sigma = \{\gamma_{\sigma,I\to F} \ | \ I, F \in \Omega\}$.
Using inequality~\eqref{eqn:mixingcongestion}, our bound on the mixing time again follows from a bound on the {\em congestion $\varrho(\Gamma_\sigma)$}, which is defined analogously to~\eqref{eqn:congestion}.

\begin{theorem}\label{thm:mainvertex,vs}
Suppose $G = (V,E)$ has vertex-separation $\ell$. Let $\sigma = (v_1,\dots,v_n)$ be an ordering of $V$ with vertex-separation at most $\ell$.  If, for some $\lambda > 0$, $\weight$ is vertex $\lambda$-multiplicative, then $\varrho(\Gamma_\sigma) \le 2n^2\hat\lambda^{4\ell}$.
\end{theorem}

\noindent
Theorem~\ref{thm:mainvertex,vs} immediately implies a good mixing time bound for the Markov chain $\Mssf$ and hence Theorem~\ref{thm:mainvertex,tw} also.

\begin{corollary}\label{cor:vs}
Let $G=(V,E)$ where $|V| = n$.  If $\weight$ is vertex $\lambda$-multiplicative for some $\lambda > 0$, then the mixing time of $\Mssf$ on $G$ satisfies
\begin{align*}
\tau(\eps) = O\left(n^2 \hat\lambda^{4\vs(G)}
\log(1/\eps)
\right).
\end{align*}
\end{corollary}

\begin{proof}
Substitute the congestion bound in~\eqref{eqn:mixingcongestion} into Theorem~\ref{thm:mainvertex,vs}.
\end{proof}

\begin{proof}[Proof of Theorem~\ref{thm:mainvertex,tw}]
Substitute the upper bound on $\vs(G)=\pw(G)$ in~\eqref{eqn:width} into Corollary~\ref{cor:vs}.
\end{proof}

We omit the proof of Theorem~\ref{thm:mainvertex,vs} as it is similar to that of Theorem~\ref{thm:mainedge,lw}, but give the details for the analogue of Lemma~\ref{lem:fdiffedge}.

\begin{lemma}\label{lem:fdiffvertex}
Suppose $G = (V,E)$ has vertex-separation $\ell$ and let $\sigma = (v_1,\dots,v_n)$ be an ordering of $V$ with vertex-separation at most $\ell$.  Suppose $I,F\in\Ossf$ and $H$ is on $\gamma_{\sigma,I\to F}$.
If $\weight$ is vertex $\lambda$-multiplicative for some $\lambda > 0$, then
\begin{align*}
\frac{\weight(G[I])\weight(G[F])}{\weight(G[H])\weight(G[C])} \le \hat\lambda^{4\ell},
\end{align*}
where $C = I \oplus F \oplus H$.
\end{lemma}

\begin{proof}
Since $H$ is on $\gamma_{\sigma,I\to F}$, we may assume that $H = H_j$ for some $j\in\{0,\dots,k\}$.  Let $Q = \{v_1,\dots,v_{i_j}\}$, and $\overline{Q} = V\setminus Q$. Then
\begin{align}
H = (F\cap Q)\cup(I\cap\overline{Q}) \quad \text{and} \quad C = (I\cap Q)\cup(F\cap\overline{Q}). \label{eqn:HCvertex}
\end{align}
We can partition $V$ into three sets as follows. Let $V_1$ denote the set of vertices $\overline{Q}$; let $V_2$ denote the subset of $Q$ containing vertices adjacent only to other vertices of $Q$; and let $K$ denote the set of remaining vertices, that is, the set of vertices of $Q$ incident to vertices of $V_1$.  Note that $|K|$ is at most the vertex-separation $\ell$.

Clearly, $K$ is a vertex cut separating $V_1$ and $V_2$ with respect to $G$ and also with respect to the graphs $G[I]$, $G[F]$, $G[H]$, $G[C]$. 
Therefore, by the fact that $\weight$ is vertex $\lambda$-multiplicative, and noting that $V_2\cup K = Q$,
\begin{align*}
\hat\lambda^{-\ell} \le \frac{\weight(G[Q\cap J])\weight(G[\overline{Q}\cap J])}{\weight(G[J])} \le \hat\lambda^\ell \quad \text{for } J\in \{I,F,H,C\}. 
\end{align*}
By~\eqref{eqn:HCvertex}, it follows that $H\cap Q = F\cap Q$, $H\cap\overline{Q} = I\cap\overline{Q}$, $C\cap Q = I\cap Q$ and $C\cap\overline{Q} = F\cap\overline{Q}$.
Now, letting $r = \weight(G[Q\cap I])\weight(G[\overline{Q}\cap I])\weight(G[Q\cap F])\weight(G[\overline{Q}\cap F])$,
we obtain that
\begin{align*}
\frac{\weight(G[I])\weight(G[F])}{r} \le \hat\lambda^{2\ell} 
\quad \text{and} \quad
\frac{r}{\weight(G[H])\weight(G[C])} \le \hat\lambda^{2\ell},
\end{align*}
whereby the lemma easily follows.
\end{proof}

\subsection{An example of a vertex subset chain}\label{subsec:exvertex}

Recalling~\eqref{eqn:interpoly}, for fixed $x,y > 1$, $\weight(G[S]) := (x-1)^{\rank_2(S)} (y-1)^{|V|-\rank_2(S)}$ gives a vertex subset weighting for $q(G;x,y)$.
With arguments very similar to those given in Subsection~\ref{subsec:kmultex}, it is not difficult to verify that this weight function is vertex $\lambda$-multiplicative.
So, by Theorem~\ref{thm:mainvertex,tw}, it follows that a natural Markov chain derived from the bivariate interlace polynomial --- a chain which has not been studied extensively, as far as we are aware --- mixes rapidly on tree-width-bounded graphs.

\begin{corollary}\label{cor:consequences,interpoly}
Let $G=(V,E)$ where $|V| = n$.  If $x,y>1$ are fixed, then for the single site flip chain on $G$ associated with the weighting~\eqref{eqn:interpoly} of $q(G;x,y)$, the mixing time satisfies
\begin{align*}
\tau(\eps) = O\left(n^{2+8(\tw(G)+1)|\log ((x-1)/(y-1))|}
\log(1/\eps)
\right).
\end{align*}
\end{corollary}

We believe that it would be of wider interest to study further properties of this single site flip chain on general graphs, in particular to compare it with known results on the random cluster, Potts and Ising models.

\section{Conclusion}

In this work, we have developed a new general framework of graph polynomials and Markov chains defined via subset expansion formulae for these polynomials, and demonstrated that their dynamics mix rapidly for graphs of bounded tree-width.  On a graph $G$ with $n$ vertices, we have shown a mixing time of order $n^{O(1)}e^{O(\pw(G))} = n^{O(\tw(G))}$.  Our results apply to many of the most prominent and well-known polynomials in the field.
The mixing times of our processes have, respectively, exponential and super-exponential dependencies upon path-width and tree-width.  We ask whether this can be improved, in particular, to achieve something akin to fixed-parameter tractability in terms of tree-width.  

\bibliographystyle{abbrv}
\bibliography{subgraphwidth}

\end{document}